\documentclass{birkjour}

\usepackage{amscd}
 \usepackage{amssymb}
 \usepackage{hyperref}
 \usepackage[all]{xypic}
 \usepackage{mathrsfs}
 \usepackage{mathtools}
 
 \usepackage[utf8]{inputenc}
\usepackage{amsmath,amssymb}
\usepackage{tikz,pgf}
\usepackage{graphicx}
\usepackage{verbatim}
 \newtheorem{theorem}{Theorem}[section]
 
 \newtheorem{proposition}[theorem]{Proposition}
 
 \newtheorem{example}[theorem]{Example}
 \theoremstyle{definition}
 \newtheorem{definition}[theorem]{Definition}
 \theoremstyle{remark}

 \numberwithin{equation}{section}

\begin{document}

%
%
%
%
%
%
%
%
%

 \title[A Characterization of Weak Proximal Normal Structure]{A Characterization of Weak Proximal Normal Structure and Best Proximity Pairs}

 \author[Abhik Digar]{Abhik Digar}

 \address{%
 Department of Mathemtics\\
 IIT Ropar\\
 Rupnagar - 140 001\\
 Punjab, India.}

\email{abhikdigar@gmail.com}

\author[Rafael Esp\'{\i}nola Garc\'{\i}a]{Rafael Esp\'{\i}nola Garc\'{\i}a}
 \address{%
 Departamento de Análisis Matemático,\\ Facultad de Matemáticas, IMUS,\\ Universidad de Sevilla, 41010 , Sevilla, Spain.}
 \email{espinola@us.es} 

 \author[G S Raju Kosuru]{G. Sankara Raju Kosuru}
 \address{%
 Department of Mathemtics\\
 IIT Ropar\\
 Rupnagar - 140 001\\
 Punjab, India.}
 \email{raju@iitrpr.ac.in}

%



%
%
%

\maketitle

\begin{abstract}
The aim of this paper is to address an open problem given in [Kirk, W. A., Shahzad, Naseer, Normal structure and orbital fixed point conditions, J. Math. Anal. Appl. {\bf{vol 463(2)}}, (2018) 461--476]. We give a characterization of weak proximal normal structure using best proximity pair property. We also introduce a notion of pointwise cyclic contraction wrt orbits and therein prove the existence of a best proximity pair in the setting of reflexive Banach spaces.
\end{abstract}

\section{Introduction and Preliminaries}
\label{intro}
Let $A, B$ be two non-empty subsets of a Banach space and $T$  be a cyclic mapping on $A\cup B$ ($T(A)\subseteq B, T(B)\subseteq A$). A pair $(x,y)\in A\times B$ is said to be a best proximity pair for $T$ if $\|x-Tx\|=\|y-Ty\|=d(A,B)=\inf\{\|x-y\|:x\in A, y\in B\}.$ The geometry of Banach spaces plays a crucial role for the existence of best proximity pairs. The analysis of proximal normal structure and weak or semi-normal structure, the property UC, the projectional property due to Eldred {\it{et al.}}\cite{Eldred 2005}, Moosa \cite{Moosa 2014}, Suzuki {\it{et al.}} \cite{Suzuki 2009}, G. S. Raju {\it{et al.}} \cite{Raju 2010} etc., respectively are widely used to prove the existence of a best proximity pair for cyclic maps.
We denote $\sup\{\|x-y\|:y\in B\}$ for $x\in A$ by $\delta(x,B).$ We shall say that the pair $(A,B)$ is proximinal pair if for every $x$ in $A$ (resp. in $B$), there exists $y$ in $B$ (resp. in $A$) such that $\|x-y\|=d(A,B).$  Further, if such a $y$ is unique, then $(A,B)$ is said to be a sharp proximinal pair \cite{Raju 2010}. In this case we denote $y$ by $x'.$ Also, $(A,B)$ is said to be a proximinal parallel pair if $(A,B)$ is sharp proximinal and $B=A+h$ for some $h\in X$ \cite{Espinola 2008}. 
 It is shown in \cite{Espinola 2008} that if $X$ is strictly convex and $A, B$ are weakly compact  convex subsets of $X,$ then every $(A_0,B_0)$ is a non-empty proximinal parallel pair. Here $A_0=\{x\in A: \mbox{there exists}~ y\in B~\mbox{such that}~\|x-y\|=d(A,B)\}$ and $B_0=\{x\in B: \mbox{there exists}~ y\in A~\mbox{such that}~\|x-y\|=d(A,B)\}.$ Also, in \cite{Raju 2010}, the authors have given example(s) of sharp proximinal pair which are not parallel. In \cite{Eldred 2005}, the author introduced a geometrical notion called proximal normal structure to prove the existence of a best proximity pair of a  relatively nonexpansive mapping ($\|Tx-Ty\|\leq \|x-y\|$ for all $x\in A, y\in B.$)
We say $(A,B)$ has proximal normal structure (\cite{Eldred 2005}) [respectively weak proximal normal structure (\cite{Moosa 2017})] if $(A,B)$ is convex and for any closed bounded [respectively weakly compact] convex proximinal pair $(H_1,H_2)$ of subsets of $(A,B)$ for which $d(H_1,H_2)=d(A,B)$ and $\delta(H_1,H_2)>d(H_1,H_2),$ there exists $(x,y)\in (H_1,H_2)$ such that $\delta(x,H_2)<\delta(H_1,H_2)$ and $\delta(y,H_1)<\delta(H_1,H_2).$ It is well known that every non-empty closed bounded convex pair $(A,B)$ of a uniformly convex Banach space has proximal normal structure. In fact, every non-empty compact convex pair $(A,B)$ of a Banach space has proximal normal structure. It is proved (Proposition 3.2 in \cite{Moosa 2017}) that a bounded convex pair has proximal normal structure if and only if it doesn't contain any proximal diametral sequence. A pair $\displaystyle\left(\{x_n\},\{y_n\}\right)$ of sequences in $(A,B)$ with $\|x_n-y_n\|=d(A,B), n\geq 1$ is said to be a proximal diametral sequence (\cite{Moosa 2017}) if $d(A,B)<\delta(\{x_n\},\{y_n\})$ and $\displaystyle \max\{\lim_{n\to \infty}d\left(x_{n+1},\mbox{co}\left(\{y_1,y_2,...,y_n\}\right)\right),\lim_{n\to \infty}d\left(y_{n+1},\mbox{co}\left(\{x_1,x_2,...,x_n\}\right)\right)\}=\delta(\{x_n\},\{y_n\}).$ It is easy to see that proximal normal structure coincides with weak proximal normal structure in reflexive Banach spaces \cite{Moosa 2017}.
Moreover, therein the author proved the existence of a best proximity pair in the settings of a reflexive Banach space. Recently, in \cite{Kirk 2018}, the authors posed an open problem for the existence of a best proximity pair for a more general class of mappings, called relatively orbital nonexpansive mappings. Also therein the authors indicated that an affirmative answer may provide a characterization of proximal normal structure. Motivated by this, we aim to give a partial  affirmative answer for the same. We also provide a characterization of weak proximal normal structure by using the existence of a best proximity pair for relatively orbital nonexpansive mappings. Finally, we introduce the notion of pointwise cyclic contraction wrt orbits and prove the minimal invariant subsets of such a map have nondiametral points. This guarantees the existence of a best proximity pair for such a class in the setting of a reflexive Banach space. Finally, we prove the existence of a best proximity pair for the class of pointwise cyclic contraction wrt orbits.


\section{Existence of Best Proximity Pairs}
\label{sec:1}
 Let $A, B$ be two closed convex subsets of a Banach space $X.$ Let $T:A\cup B\to A\cup B$ be a cyclic map. If $T$ admits a best proximity pair, then $A_0\neq \emptyset \neq B_0.$ Also, if $T$ is relatively nonexpansive, then $A_0\cup B_0$ is cyclically invariant under $T$ ($TA_0\subseteq B_0, TB_0\subseteq A_0$). The following theorem is due to Eldred $\it{et~al.}$ \cite{Eldred 2005} 
 \begin{theorem}\label{Eldred}
 Let $(K_1,K_2)$ be a non-empty weakly compact convex pair in a Banach space and suppose $(K_1,K_2)$ has proximal normal structure. Then every relatively nonexpansive mapping $T$ on $A\cup B$ has a best proximity pair in $(K_1,K_2).$
 \end{theorem}
 
 The main tool to prove the same is to use the geometrical notion called ``proximal normal structure" on $A_0\cup B_0$. Later many authors established the existence of a best proximity pair for relatively nonexpansive mappings in different settings using variants of geometry (\cite{Espinola 2008},\cite{Moosa 2014},\cite{Raju 2011},\cite{Suzuki 2009}). In \cite{Moosa 2017}, Moosa introduced pointwise relatively nonexpansive mappings involving orbits and therein proved the existence of a best proximity pair for such a class of mappings. Recently, in 2018, Kirk and Shahzad discussed the existence of a best proximity pair for relatively nonexpansive mappings and therein they raised the question ``can the assumption that $T$ is relatively nonexpansive in Theorem \ref{Eldred} be replaced by the
assumption that $T$ is relatively nonexpansive wrt orbits?" $T$ is said to be relatively nonexpansively mappings wrt orbits if $\|Tx-Ty\|\leq r_x\left(\mathcal{O}(y)\right)=\delta(x,\{y,Ty,T^2y,...\}).$ Using the following example, we can conclude that the answer is negative for the above open problem.
\begin{example}
 Let $A=\{x\in \mathbb{R}:-2\leq x \leq -1\}, B=\{x\in \mathbb{R}:1\leq x \leq 2\}.$ Define
 
 \[
  T(x) = 
  \begin{cases}
    -x, & \text{if } x \in A \\
    -1-\frac{x}{2}, & \text{if } x\in B.
  \end{cases}
 \]
Let $y\in B.$ For any $n,~T^{2n}y=1+\frac{2^{n-1}-1}{2^{n-1}}+\frac{y}{2^n}=2-\frac{1}{2^{n-1}}+\frac{y}{2^n}$ and $T^{2n+1}y=-1-\frac{T^{2n}y}{2}=-2+\frac{1}{2^n}-\frac{y}{2^{n+1}}.$ Now, for any $x\in A, y\in B,~\left\|Tx-Ty\right\|=\left|(-x)-\left(-1-\frac{y}{2}\right)\right|\leq 2-x=r_x\left(\mathcal{O}(y)\right).$
 \end{example} 
 It is to be observed that a cyclic map $T$ on $A\cup B$ that satisfies $\|Tx-Ty\|\leq r_x\left(\mathcal{O}(y)\right)$ does not guarantee $A_0\cup B_0$ is cyclically invariant under $T.$ Hence, it is not reasonable to expect the existence of a best proximity pair for such a map $T.$ To overcome this, we redefine the relatively orbital nonexpansive mappings. For $x\in A\cup B,$ we denote $\{T^{2n}x:n\in \mathbb{N}\cup \{0\}\}$ by $\mathcal{O}^2(x)$.
 \begin{definition}
  Let $A,B$ be two non-empty subsets of a Banach space $X.$ A cyclic map $T:A\cup B\to A\cup B$ is said to be a relatively orbital nonexpansive mapping if
 \begin{itemize}
 \item[(i)] $\|Tx-Ty\|=d(A,B)$ if $\|x-y\|=d(A,B)$ for $x\in A, y\in B.$
 \item[(ii)]  for all $x\in A, y\in B,~\|Tx-Ty\|\leq \min \{r_x\left(\mathcal{O}^2(y)\right), r_y\left(\mathcal{O}^2(x)\right)\}$.
 \end{itemize} 
 \end{definition}
 It is worth mentioning that relatively orbital nonexpanive mapping is not necessarily relatively nonexpansive.
 \begin{example}\label{Example2}
 Let $A=\{(0,x)\in \mathbb{R}^2:0\leq x\leq 1\}, B=\{(1,y)\in \mathbb{R}^2:0\leq y\leq 1\}$ and $T:A\cup B \to A\cup B$ be defined by
 
 \[
  x\in A,~ T(x)=
   \begin{cases}
   (1,\frac{x}{4})  & \text{if}~  x\geq \frac{1}{2};\\
   (1,\frac{x}{2})  & \text{if}~  x< \frac{1}{2}.
  \end{cases} 
 \]
 
 \[
   y\in B,~T(y)=
   \begin{cases}
   (0,\frac{y}{4})  & \text{if}~  y\geq \frac{1}{2};\\
   (0,\frac{y}{2})  & \text{if}~  y< \frac{1}{2}.
  \end{cases} 
 \]
 
 We see that $T$ is not relatively nonexpansive but relatively orbital nonexpansive mapping.
  \end{example}
  Let $(A, B)$  be a non-empty sharp proximinal pair in Banach space and $T$ be a relatively orbital nonexpansive mapping on $A\cup B.$ Then it is easy to see that $(A_0,B_0)$ is cyclically invariant under $T$ and $Tx'=(Tx)'.$ We say that  $(A,B)$  is said to satisfy the weak best proximity pair property (WBPP) if every relatively orbital nonexpansive mapping on $A\cup B$ has a best proximity pair.
  The following theorem ensures that every non-empty weakly compact convex pair of subsets of a strictly convex Banach space satisfying the WBPP. The following theorem is in a way different than Theorem 2.6 of \cite{Moosa 2017}. For the sake the completeness, we prove the same here. 
  \begin{theorem}\label{first Theorem}
Let $A,B$ be two non-empty weakly compact convex substes of a Banach space $X$. If $(A,B)$ is a sharp proximinal pair having weak proximal normal structure, then $(A, B)$ has WBPP.
\end{theorem}

 \begin{proof}
 Let $\mathscr{F}$ denote the collection of non-empty closed bounded convex proximinal pair $(E_1, E_2$) of subsets of $(A_0, B_0)$ with $(E_1, E_2)$ cyclically invariant under $T$ and $d(E_1,E_2)=d(A,B).$ 
$\mathscr{F}\neq \emptyset,$ since $(A_0,B_0)\in \mathscr{F}.$ By Zorn's Lemma $\mathscr{F}$ has a minimal element under the set inclusion order $``\subseteq"$, say, $(F_1,F_2).$ If $(F_1,F_2)$ is a singleton pair, we have $\delta(F_1,F_2)=d(A,B),$ i.e., $T$ has a best proximity pair. Suppose $(F_1,F_2)$ is not singleton. By weak proximal normal structure, there exist $(x_1,y_1)\in (F_1,F_2)$ such that $ m_1 = \delta(x_1,F_2)<\delta(F_1,F_2);~
 m_2 = \delta(y_1,F_1)<\delta(F_1,F_2).$ Set $m=\max\{m_1,m_2\}$. Define
 \begin{eqnarray*}
 L_1&=&\{x\in F_1:\delta(x,F_2)\leq m\}\\
 L_2&=&\{y\in F_2:\delta(y,F_1)\leq m\}.
 \end{eqnarray*}
 $L_1\neq \emptyset, L_2\neq \emptyset,$ since $x_1\in L_1,y_1\in L_2.$ Being closed subset of a weakly compact subset, $ L_1,L_2$ are weakly compact. To see $L_1$ is convex, let $a, b\in L_1.$ For any $\lambda\in [0,1],\\ 
 \delta\left(\lambda a+(1-\lambda)b,F_2\right)\leq \lambda\delta(a,F_2)+(1-\lambda)\delta(b,F_2) \leq \lambda m+(1-\lambda)m=m.$ Hence we can conclude that $(L_1,L_2)$ is a convex pair. Let $v\in F_2.$ Suppose the unique best approximation of an element $z\in A\cup B$ is denoted by $z'$. Then 
 \begin{eqnarray*}
 \left\|\frac{x_1+y_1'}{2}-v\right\| &\leq &\frac{1}{2} \left[\left\|x_1-v\right\|+\left\|y_1'-v\right\|\right] \\ &=& \frac{1}{2}\left[\left\|x_1-v\right\|+\left\|y_1-v'\right\|\right]\\
&\leq & \frac{1}{2}\left[\delta(x_1,F_2)+\delta(y_1,F_1) \right]\\  &\leq & m.
 \end{eqnarray*}

 Since $v\in F_2$ is arbitrary, $\delta\left(\frac{x_1+y_1'}{2},F_2\right)\leq m.$ Hence, $\frac{x_1+y_1'}{2}\in L_1$ Similarly, $\frac{x_1'+y_1}{2}\in L_2.$ Moreover, $\left\|\frac{x_1+y_1'}{2}-\frac{x_1'+y_1}{2}\right\|=d(A,B).$ Hence, $d(L_1,L_2)=d(A,B).$ To see $(L_1,L_2)$ is a proximinal pair, let $x\in L_1.$ Then $x\in F_1$ and hence $x'\in F_2.$ Therefore $\delta(x',F_1)=\delta(x,F_2)\leq m.$ Thus $x'\in L_2.$ It infers $(L_1,L_2)$ is a proximinal pair. Thus, $L_2=\{x'\in F_2:x\in L_1\}.$

 Next, let $x\in L_1, v\in F_2.$ Then, $\left\|Tx-Tv\right\|\leq r_x\left(\mathcal{O}^2(v)\right)=\delta\left(x,\mathcal{O}^2(v)\right)\leq \delta(x,F_2)\leq m.$ It follows that $T(F_2)\subset B\left(Tx;m\right)\cap F_1=F_1'.$ Similarly, $T(F_1)\subset B\left(Tx';m\right)\cap F_2=F_2'.$ Clearly, $(F_1', F_2')\in \mathscr{F}.$ By minimality, $F_1'=F_1, F_2'=F_2.$ Then $F_1\subseteq B(Tx;m)$ and $F_2\subseteq B(Tx';m).$ For any $u\in F_1, \|u-Tx\|\leq m,$ hence, $\delta(Tx,F_1)\leq m.$ Therefore, $Tx\in L_2.$ Hence, $T(L_1)\subseteq L_2.$
%
%
%
 Further, if $y\in L_2,$ then $y'\in L_1.$ This implies $Ty'=(Ty)'\in L_2.$ Thus $Ty\in L_1.$
As $y\in L_2$ is arbitrary, we have $T(L_2)\subseteq L_1.$ Hence, $(L_1,L_2)\in \mathscr{F}.$ For $x\in L_1, y\in L_2, \left\|x-y\right\|\leq \delta(x,F_2)\leq m<\delta(F_1,F_2).$ This infers that $\delta(L_1,L_2)<\delta(F_1,F_2).$ This contradicts the minimality of $(F_1,F_2).$
\end{proof}

 Let $T$ be a cyclic map on $A\cup B$. We say that the pair $(A,B)$ has a proximinal nondiametral pair if there exists $(x,y)\in A\times B$ such that $\max\{\delta(x,B),\delta(y,A)\}<\delta(A, B)$ whenever $d(A,B)<\delta(A,B).$ 
 A similar technique can be used to obtain the following:
 
 \begin{theorem}\label{Theorem 3.3}
 Let $(A, B)$ be a non-empty closed bounded convex proximinal pair of subsets of a Banach space and let $T$ be a relatively orbital nonexpansive mapping on $A\cup B.$ If $T$ has a nonempty closed bounded convex minimal cyclically invariant pair $(A,B)$ having a nondiametral pair then $T$ has a best proximity pair.
 \end{theorem}

 \begin{example}
 Let $A, B$ and $T$ as in the Example \ref{Example2}. It is easy to see that $\left((0,0),(1,0)\right)$ is a best proximity pair.
 \end{example}


 \section{Characterization of weak proximal normal structure}

  Let $(A,B)$ be a bounded convex proximinal pair of a Banach space $X.$ A non-constant pair of sequences $\left(\{x_n\},\{y_n\}\right)$ of $(A,B)$ is said to be a proximinal diametral sequence if $\|x_n-y_n\|=d(A,B)$ for every $n\in \mathbb{N}$ and $\delta(\{x_n\},\{y_n\})= \displaystyle \lim_{n\to \infty}d\left(x_{n+1},\mbox{co}\left(\{y_1,y_2,...,y_n\}\right)\right) =\lim_{n\to \infty}d\left(y_{n+1},\mbox{co}\left(\{x_1,x_2,...,x_n\}\right)\right).$ It is to be observed that if $d(A,B)=0,$ then the proximinal diametral sequence turns out to be a diametral sequence in $A\cap B$ in the sense of Brodskii and Milman (\cite{Milman 1948}). Using a similar argument employed in the proof of Theorem 2.5 (\cite{Eldred 2005}) one can obtain the following: 
 \begin{theorem}\label{Theorem3.1}
 A bounded convex pair $(A,B)$ of a Banach space $X$ has proximal normal structure if and only if it does not contain a proximinal diametral sequence.
 \end{theorem}

  Let $(A, B)$ be a non-empty weakly compact convex sharp proximinal pair of subsets of a Banach space having WBPP. Suppose $(A, B)$ does not have proximal weak normal structure. Then by  Theorem \ref{Theorem3.1}, $(A, B)$ has a proximinal diametral sequence, say, $\left(\{x_n\},\{y_n\}\right).$ Consequently, $\displaystyle\lim_{n\to \infty}d\left(x_{n+1}, \mbox{co}\left(\{y_1,y_2,...,y_n\}\right)\right)=\delta(\{x_n\},\{y_n\})=\lim_{n\to \infty}d\left(y_{n+1}, \mbox{co}\left(\{x_1,x_2,...,x_n\}\right)\right).$
 
 Since, $(A, B)$ is weakly compact, there exists a subsequence $\left(\{x_{n_{k}}\},\{y_{n_{k}}\}\right)$ of $\left(\{x_n\},\{y_n\}\right)$ which is weakly convergent. It is easy to see that the sequence $\left(\{x_{n_{k}}\},\{y_{n_{k}}\}\right)$ is a proximinal diametral subsequence. Hence, without loss of any generality, we may assume that the sequence $\left(\{x_n\},\{y_n\}\right)$ is  proximinal diametral and weakly convergent.
 Now, $H=\overline{\mbox{co}}\left(\{x_1,x_2,...\} \right),K= \overline{\mbox{co}}\left(\{y_1,y_2,...\} \right)$ are weakly compact convex subsets of $A, B$ respectively. Define $T:H\cup K \to H\cup K$ by
 
 \[
  T(x) = 
  \begin{cases}
    y_1, & \text{if} ~x \notin \{x_n:n\in \mathbb{N}\} \\
    y_{n+1}, & \text{if}~ x= x_n ~ \text{for~some}~ n\in \mathbb{N};
  \end{cases}
\]

 \[
  T(y) = 
  \begin{cases}
    x_1, & \text{if} ~y \notin \{y_n:n\in \mathbb{N}\} \\
    x_{n+1}, & \text{if}~ y= y_n ~ \text{for~some}~ n\in \mathbb{N}.
  \end{cases}
\]
 
 Clearly, $\delta(H,K)=\delta(\{x_n\},\{y_n\})$ and $\displaystyle\lim_{n\to \infty}\|x_n-z\|=\delta(H,K)=\lim_{n\to \infty}\|y_n-v\|$ for any $z\in K, v\in H.$ Hence, $r_x\left(\mathcal{O}^2(y)\right)=\delta(H,K)$ for each $x\in H, y\in K.$ Now, $$\|Tx-Ty\|\leq \delta(H,K)=r_x\left(\mathcal{O}^2(y)\right)~~\mbox{for each}~x\in H, y\in K.$$ Also, if $(x, y)\in H\times K$ with $\|x-y\|=d(H, K),$ then $\|Tx-Ty\|=d(H,K).$ Therefore $T$ is a relatively orbital nonexpansive mapping. As $(A, B)$ is a sharp proximinal pair, then so is $(H, K)$ and $T$ does not have any best proximity pair. Thus we have the following: 
 
  \begin{proposition}\label{Proposition3.1}
Let $A,B$ be two non-empty weakly compact convex substes of a Banach space $X$. If $(A,B)$ is a sharp proximinal pair and $(A, B)$ has WBPP, then $(A, B)$ has weak proximal normal structure.
\end{proposition} 
 
 By Theorem \ref{first Theorem} and Proposition \ref{Proposition3.1} we have the following characterization:
 
 \begin{theorem}
Let $A,B$ be two non-empty weakly compact convex substes of a Banach space $X$. If $(A,B)$ is a sharp proximinal pair, then $(A,B)$ has weak proximal normal structure if and only if every relatively orbital nonexpansive mapping $T:A\cup B\to A\cup B$ has a best proximity pair. 
\end{theorem}

 \section{Pointwise Cyclic Contraction wrt Orbits}
 Let $(A, B)$ be a pair of subsets of a normed linear space. A cyclic map $T$ on $A\cup B$ is said to be a proximal pointwise contraction if for any $x\in A,$ there exists $\alpha (x)\in [0,1)$ such that $\|Tx-Ty\|\leq \alpha(x)\|x-y\|$ (\cite{Anuradha 2009}). 
Later many authors obtained the existence of a best proximity pair for certain types of pointwise cyclic contractions (\cite{Raju 2011}, \cite{Poom 2013}, \cite{Moosa 2018}).  Now we introduce the notion of pointwise cyclic contraction wrt orbits and prove the existence of a best proximity pair for such a map. 
Our result is a generalization of the main results given in the aforementioned articles.
 \begin{definition}
 A cyclic map $T$ on a non-empty pair $(A, B)$ of subsets of a Banach space is said to be pointwise cyclic contraction wrt orbits if it satisfies
 \begin{itemize}
  \item[(i)] $\|Tx-Ty\|=d(A,B)$ whenever $\|x-y\|=d(A,B)$ for  $(x,y)\in A\times B$ ;
  \item[(ii)] for each $(x,w)\in (A,B)$ there exists $\alpha(x), \alpha(w)\in (0,1)$ such that\\ $\|Tx-Ty\|\leq \alpha (x)r_x\left(\mathcal{O}^2(y)\right)+\left(1-\alpha(x)\right)d(A,B)$ for all $y\in B$, and \\
    $\|Tw-Tu\|\leq \alpha (w)r_w\left(\mathcal{O}^2(u)\right)+\left(1-\alpha(w)\right)d(A,B)$ for all $u\in A.$
 \end{itemize}

 \end{definition}
 
 It is easy to see that every pointwise cyclic contraction mapping wrt orbits is relatively orbital nonexpansive.
%
 \begin{theorem}
 Suppose $(A,B)$ is a closed, weakly compact, convex, sharp proximinal pair of a Banach space $X$ and $T:A\cup B\to A\cup B$ is a pointwise cyclic contraction wrt orbits. Then $T$ has a best proximity pair.
 \end{theorem}
 
 \begin{proof}
 Let $\mathscr{F}$ denote the collection of all non-empty proximal closed convex subsets $(H_1,H_2)$ of $(A_0,B_0)$ such that $TH_1\subseteq H_2, TH_2\subseteq H_1$ and $d(H_1,H_2)=d(A,B).$ Since $A_0\cup B_0\in \mathscr{F},$ we have $\mathscr{F}\neq \emptyset$. By Zorn's lemma, $\mathscr{F}$ has a minimal, say, $(K_1,K_2).$ Let $(x,y)\in (K_1,K_2)$ such that $\|x-y\|=d(K_1,K_2)=d(A,B).$ If $\delta(x,K_2)=d(A,B),$ then $d(A,B)=d(K_1,K_2)\leq \|x-Tx\|\leq \delta(x,K_2)=d(A,B).$ This infers $\|x-Tx\|=d(A,B).$ Since, $T$ is pointwise cyclic contraction wrt orbits, we have $\|Tx-T^2x\|=d(A,B).$ Therefore, $(x,Tx)$ is a best proximity pair. Similarly, if $\delta(y,K_1)=d(A,B),$ then $(y,Ty)$ is a best proximity pair. Hence, we may assume that $\delta(x,K_2)>d(A,B)$ and $\delta(y,K_1)>d(A,B)$. Define
 \begin{eqnarray*}
 K_x &=& \left\{z\in K_1: \|z-Tx\|\leq \alpha(x)\delta(x,K_2)+\left(1-\alpha(x)\right)d(A,B)\right\};\\
 K_y &=&\left\{w\in K_2: \|w-Ty\|\leq \alpha(y)\delta(y,K_1)+\left(1-\alpha(x)\right)d(A,B)\right\}.
 \end{eqnarray*}  
 
 Since
 \begin{eqnarray*}
 \|Tx-Ty\|= d(A,B)&= & \alpha(x)d(A,B)+(1-\alpha(x))d(A,B)\\ &< & \alpha(x) \delta(x,K_2)+(1-\alpha(x))d(A,B).
 \end{eqnarray*}
Then $(Ty,Tx)\in (K_x,K_y)$ and hence $K_x\neq \emptyset \neq K_y.$ It is easy to see that $(K_x,K_y)$ is convex. If $\{u_n\}_{n=1}^{\infty}\subset K_x$ is a sequence converges to $u\in X$ weakly, then $u\in K_1.$ Now, $\|u-Tx\|\leq \liminf \{\|u_n-Tx\|:n \in \mathbb{N}\}\leq \alpha(x)\delta(x,K_2)+(1-\alpha(x))d(A,B).$ Then $u\in K_x$ and $K_x$ is closed. Further, for any $u\in K_x,~\|Tu-Ty\|\leq \alpha(y)r_y\left(\mathcal{O}^2(u)\right)+(1-\alpha(y))d(A,B)\leq \alpha(y)\delta(y,K_1)+(1-\alpha(y))d(A,B)$. This implies that $Tu\in K_y.$ Hence, $TK_x\subseteq K_y$. Similarly, $TK_y\subseteq K_x$. Therefore, $(K_x, K_y)\in \mathscr{F}.$ By minimality, $K_x=K_,~K_y = K_2$. Now, for any $w\in K_2, \|w-Ty\|\leq \alpha(y)\delta(y,K_1)+(1-\alpha(y))d(A,B)<\delta(y,K_1)\leq \delta(K_1,K_2).$ Hence, $\delta(Ty, K_2)<\delta(K_1,K_2).$ Similarly, $\delta(Tx, K_1)<\delta(K_1,K_2).$ Thus $(K_1, K_2)$ has a  proximinal nondiametral pair. By Theorem \ref{Theorem 3.3}, T has a best proximity pair.

 \end{proof}

\end{document}